\documentclass[11pt,reqno]{amsart}

\usepackage{xspace}
\usepackage{amsmath,amsthm,amsfonts,amssymb,latexsym,mathrsfs,color,extarrows}
\usepackage{hyperref}
\usepackage{tikz}

\newcommand*\circled[1]{\tikz[baseline=(char.base)]{
  \node[shape=circle,draw,inner sep=0pt] (char) {#1};}}

\newtheorem{theorem}{Theorem}
\newtheorem{cor}[theorem]{Corollary}
\newtheorem{prop}[theorem]{Proposition}

\newtheorem{ex}[theorem]{Example}
\newtheorem{lemma}[theorem]{Lemma}
\newtheorem{definition}[theorem]{Definition}

\usepackage{color}

\newcommand{\casc}{{\rm casc\,}}
\newcommand{\mcq}{\mathcal{CQ}}

\newcommand{\desi}{{\rm desi\,}}
\newcommand{\NBar}{{\rm NBar\,}}
\newcommand{\hhat}{{\rm hat\,}}
\newcommand{\HHat}{{\rm Hat\,}}
\newcommand{\bbar}{{\rm bar\,}}
\newcommand{\BBar}{{\rm Bar\,}}
\newcommand{\redd}{{\rm red\,}}
\newcommand{\caplat}{{\rm cap\,}}
\newcommand{\cplat}{{\rm cplat\,}}

\newcommand{\ap}{{\rm ap\,}}
\newcommand{\asc}{{\rm asc\,}}
\newcommand{\fix}{{\rm fix\,}}
\newcommand{\cyc}{{\rm cyc\,}}

\newcommand{\rlmin}{{\rm rlmin\,}}
\newcommand{\RLMIN}{{\rm RLMIN\,}}
\newcommand{\exc}{{\rm exc\,}}
\newcommand{\des}{{\rm des\,}}

\newcommand{\el}{{\rm el\,}}
\newcommand{\ol}{{\rm ol\,}}

\newcommand{\mdq}{\mathcal{DQ}}
\newcommand{\mq}{\mathcal{Q}}
\newcommand{\mbn}{{\mathcal B}_n}
\newcommand{\mdqn}{\mathcal{DQ}_n}

\newcommand{\mdn}{\mathcal{D}}
\newcommand{\mqn}{\mathcal{Q}_n}
\newcommand{\mqnn}{\mathcal{Q}_{n+1}}

\newcommand{\msn}{\mathfrak{S}_n}
\newcommand{\msss}{\mathfrak{S}}

\newcommand{\z}{ \mathbb{Z}}

\newcommand{\msp}{{\mathcal P}}
\newcommand{\mbp}{{\mathcal B}}
\newcommand{\I}{{\rm I}}

\newcommand{\mmn}{\mathcal{M}_{2n}}

\newcommand{\m}{{\rm M}}
\newcommand{\mm}{\mathcal{M}}

\newcommand{\PM}{\mathcal{PM}}
\newcommand{\BM}{\mathcal{BM}}

\newcommand{\Eulerian}[2]{\genfrac{<}{>}{0pt}{}{#1}{#2}}

\newcommand{\arxiv}[1]{\href{http://arxiv.org/abs/#1}{\texttt{arXiv:#1}}}

\title[Eulerian polynomials and Stirling permutations of the second kind]{Eulerian polynomials, perfect matchings and Stirling permutations of the second kind}
\author[S.-M.~Ma]{Shi-Mei Ma}
\address{School of Mathematics and Statistics,
        Northeastern University at Qinhuangdao,
         Hebei 066004, P.R. China}
\email{shimeimapapers@163.com (S.-M. Ma)}
\author[Y.-N. Yeh]{Yeong-Nan Yeh}
\address{Institute of Mathematics,
        Academia Sinica, Taipei, Taiwan}
\email{mayeh@math.sinica.edu.tw (Y.-N. Yeh)}
\subjclass[2010]{Primary 05A15; Secondary 05A19}

\begin{document}

\maketitle
\begin{abstract}
In this paper, we first present combinatorial proofs of a kind of expansions of the Eulerian polynomials of types $A$ and $B$, and
then we introduce Stirling permutations of the second kind. In particular, we count Stirling permutations of the second kind by their cycle ascent plateaus, fixed points and cycles.
\bigskip\\
{\sl Keywords:} Eulerian polynomials; Perfect matchings; Stirling permutations of the second kind; Stirling derangements
\end{abstract}
\date{\today}
\section{Introduction}
Let $\msn$ be the symmetric group on the set $[n]=\{1,2,\ldots,n\}$ and
let $\pi=\pi(1)\pi(2)\cdots\pi(n)\in\msn$.
Denote by $\mbn$ the hyperoctahedral group of rank $n$. Elements $\pi$ of $\mbn$ are signed permutations of the set $\pm[n]$ such that $\pi(-i)=-\pi(i)$ for all $i$, where $\pm[n]=\{\pm1,\pm2,\ldots,\pm n\}$.
Let $\#S$ denote the cardinality of a set $S$.
We define
\begin{equation*}
\begin{split}
\des_A(\pi)&:=\#\{i\in\{1,2,\ldots,n-1\}|\pi(i)>\pi({i+1})\},\\
\des_B(\pi)&:=\#\{i\in\{0,1,2,\ldots,n-1\}|\pi(i)>\pi({i+1})\},
\end{split}
\end{equation*}
where $\pi(0)=0$.
The Eulerian polynomials of types $A$ and $B$  are
respectively defined by
\begin{equation*}
\begin{split}
A_n(x)&=\sum_{\pi\in\msn}x^{\des_A(\pi)},\\
B_n(x)&=\sum_{\pi\in\mbn}x^{\des_B(\pi)}.
\end{split}
\end{equation*}
There is a larger literature devoted to $A_n(x)$ and $B_n(x)$ (see, e.g.,~\cite{Bre94,Bre00,Dilks09,Foata09,Lin15,Savage1201} and references therein).
Let $s=(s_1,s_2,\ldots)$ be a sequence of positive integers. Let
$$I_n^{(s)}=\left\{ (e_1,e_2,\ldots,e_n)\in \z^n|~0\leq e_i< s_i\right\},$$ which known as the set of $s$-inversion sequences.
The number of {\it ascents} of an $s$-inversion sequence $e=(e_1,e_2,\ldots,e_n)\in I_n^{(s)}$ is defined by
$$\asc(e)=\#\left\{i\in [n-1]:\frac{e_i}{s_i}<\frac{e_{i+1}}{s_{i+1}}\right\}\cup\{0: \textrm{if $e_1> 0$}\}.$$
Let $E_n^s(x)=\sum_{e\in I_n^s}x^{\asc(e)}$.
Following~\cite{Savage1201}, we have
\begin{align*}
A_n(x)&=E_n^{(1,2,\ldots,n)}(x), \\
B_n(x)&=E_n^{(2,4,\ldots,2n)}(x).
\end{align*}

Let $M_n(x)$ be a sequence of polynomials defined by
\begin{equation}\label{N2xt02}
M(x,z)=\sum_{n\geq 0}M_n(x)\frac{z^n}{n!}=\sqrt{\frac{x-1}{x-e^{2z(x-1)}}}.
\end{equation}
Combining~\eqref{N2xt02} and an explicit formula of the Ehrhart polynomial of the $s$-lecture hall polytope, Savage and Viswanathan~\cite{Savage1202} proved that
$M_n(x)=E_n^{(1,3,\ldots,2n-1)}(x)$.

A {\it perfect matching} of $[2n]$ is a partition of $[2n]$ into $n$ blocks of size $2$. Denote by $N({n,k})$ the number of perfect matchings of $[2n]$ with the restriction that only $k$ matching pairs have even larger entries.
The numbers $N(n,k)$ satisfy the recurrence relation
\begin{equation*}
N(n+1,k)=2kN(n,k)+(2n-2k+3)N(n,k-1)
\end{equation*}
for $n,k\geq 1$, where $N(1,1)=1$ and $N(1,k)=0$ for $k\geq 2$ or $k\leq 0$ (see~\cite[Proposition~1]{MaYeh}).
Let $N_n(x)=\sum_{k=1}^nN({n,k})x^k$. The first few of the polynomials $N_n(x)$ are
$$N_0(x)=1,N_1(x)=x, N_2(x)=2x+x^2, N_3(x)=4x+10x^2+x^3.$$
The exponential generating function for $N_n(x)$ is given as follows (see~\cite[Eq.~(25)]{Ma13}):
\begin{equation}\label{N2xt}
N(x,z)=\sum_{n\geq 0}N_n(x)\frac{z^n}{n!}=\sqrt{\frac{1-x}{1-xe^{2z(1-x)}}}.
\end{equation}
Combining~\eqref{N2xt02} and~\eqref{N2xt}, we get $M_n(x)=x^{n}N_n(\frac{1}{x})$ for $n\geq 0$.

Context-free grammar was introduced by Chen~\cite{Chen93} and it is a powerful tool for studying exponential structures in combinatorics.
We refer the reader to~\cite{Chen121,Chen14,Dumont96,Ma-EuJC} for further information. In particular, using~\cite[Theorem~10]{Ma-EuJC}, it is easy to present a grammatical proof of the following result.
\begin{prop}\label{prop01}
For $n\geq 0$, we have
\begin{equation}\label{NnxAnx}
2^nxA_n(x)=\sum_{k=0}^n\binom{n}{k}N_k(x)N_{n-k}(x),
\end{equation}
\begin{equation}\label{NnxBnx}
B_n(x)=\sum_{k=0}^n\binom{n}{k}N_k(x)M_{n-k}(x).
\end{equation}
\end{prop}

Recall that the exponential generating function for $xA_n(x)$ is
\begin{equation*}\label{Axz}
A(x,z)=1+\sum_{n\geq 1}xA_n(x)\frac{t^n}{n!}=\frac{1-x}{1-xe^{z(1-x)}}.
\end{equation*}
An equivalent formula of~\eqref{NnxAnx} is given as follows:
\begin{equation}\label{N2xzAx2z}
N^2(x,z)=A(x,2z).
\end{equation}

One purpose of this paper is to study the correspondence between permutations and pairs of perfect matchings.
Motivated by~\eqref{N2xzAx2z}, another purpose of this paper is to explore some cycle structure related to $N(x,z)$ or $M(x,z)$.
This paper is organized as follows. In Section~\ref{Section02}, we present a combinatorial proof of Proposition~\ref{prop01}.
In Section~\ref{Section03}, we introduce the Stirling permutations of the second kind.
In Section~\ref{Section04}, we count Stirling permutations of the second kind by their cycle ascent plateaus, fixed points and cycles.
\section{A combinatorial proof of Proposition~\ref{prop01}}\label{Section02}
Let $\mmn$ be the set of perfect matchings of $[2n]$, and let $\m\in\mmn$.
The \emph{standard form} of $\m$ is a list of blocks $\{(i_1,j_1),(i_2,j_2),\ldots,(i_n,j_n)\}$ such that
$i_r<j_r$ for all $1\leq r\leq n$ and $1=i_1<i_2<\cdots <i_n$.
In the following discussion we always write $\m$ in standard form.
Let $\el(\m)$ (resp. $\ol(\m)$) be the number of blocks of $\m$ with even larger (resp. odd larger) entries.
Therefore, we have
$$N_n(x)=\sum_{\m\in\mmn}x^{\el(\m)},$$
$$M_n(x)=\sum_{\m\in\mmn}x^{\ol(\m)}.$$
For convenience, we call $(i,j)$ a \emph{marked block} (resp.~an \emph{unmarked block}) if $j$ is even (resp. odd) and
large than $i$.

\subsection{Permutations and pairs of perfect matchings}
\hspace*{\parindent}

Let the entry $\pi(i)$ be called a {\it descent} (resp. an {\it ascent}) of $\pi$ if $\pi(i)>\pi(i+1)$ (resp. $\pi(i)<\pi(i+1)$).
By using the reverse map, it is evident that the ascent and descent statistics are equidistributed.
Let $\asc(\pi)$ be the number of ascents of $\pi$. Hence
\begin{equation}\label{Anx-asc}
A_n(x)=\sum_{\pi\in\msn}x^{\asc(\pi)}.
\end{equation}
Throughout this subsection, we shall always use~\eqref{Anx-asc} as the definition of $A_n(x)$.

We now constructively define a set of decorated permutations on $[n]$ with some entries of permutations decorated with hats and circles, denoted by $\msp_n$. Let $w=w_1w_2\cdots w_n\in \msp_n$. We say that $w_i$ with a hat (resp. circle) if
$w_i=\tiny{\widehat{{k}}}$ or $w_i={\circled{$\widehat{{k}}$}}$ (resp. $w_i={\circled{$k$}}$ or $w_i={\circled{$\widehat{{k}}$}}$) for some $k\in [n]$.
Start with $\msp_1=\{1,\widehat{{1}}\}$.
Suppose we have get $\msp_{n-1}$, where $n\geq 2$.
Given $v=v_1v_2\cdots v_{n-1}\in\msp_{n-1}$. We now construct
entries of $\msp_{n}$ by inserting $n,{\circled{n}},\widehat{{n}}$ or $\circled{$\widehat{n}$}$ into $v$ according the following rules:
\begin{enumerate}
\item [($r_1$)] We can only put $n$ or $\widehat{n}$ at the end of $v$;
\item [($r_2$)] For $1\leq i\leq n-1$, if $v_i$ with no bar, then we can only put $n$ or ${\circled{n}}$ immediately before $v_i$;
if $v_i$ with a bar, then we can only put
$\widehat{n}$ or $\circled{$\widehat{n}$}$ immediately before $v_i$. In other words, if $v_i$ with a hat (resp. with no hat), then we can only insert $n$ with a hat (resp. with no hat) immediately before $v_i$.
\end{enumerate}

It is clear that there are $2n$ elements in $\msp_{n}$ that can be generated from any $v\in\msp_{n-1}$.
By induction, we obtain $|\msp_n|=2n|\msp_{n-1}|=2^nn!$.
Let $\varphi(w)=\varphi(w_1)\varphi(w_2)\cdots\varphi(w_n)$ be a permutation of $\msn$ obtained from $w\in\msp_n$ by deleting the hats and circles of all $w_i$.
For example, $\varphi(\widehat{3}~\widehat{1}\circled{4}~2)=3142$.
Let $\msp_{n}(\pi)=\{w\in\msp_n: \varphi(w)=\pi\}$. Let $k\ell$ be a consecutive subword of $\pi\in\msn$.
By using the above rules, we see that
if $k<\ell$, then $k\ell$ can be decorated as follows:
$$k\ell,k\widehat{\ell},\widehat{k}\ell,\widehat{k}~\widehat{\ell};$$
If $k>\ell$, then $k\ell$ can be decorated as
$k\ell,\circled{k}\ell,\widehat{k}~\widehat{\ell},\circled{$\widehat{k}$}~\widehat{\ell}$.
Therefore, $|\msp_n(\pi)|=2^n$ for any $\pi\in\msn$. It should be noted that $k\widehat{\ell}$ or $\widehat{k}\ell$ is a consecutive subword of $w\in\msp_n$ if and only if $k<\ell$.
Let the entry $w_i$ be called an {\it ascent} (resp. a {\it descent}) of $w$ if
$\varphi(w_i)<\varphi(w_{i+1})$ ($\varphi(w_i)>\varphi(w_{i+1})$). Also a conventional ascent is counted at the beginning of $w$.
That is, we identify a decorated permutation $w=w_1\cdots w_n$ with the word $w_0w_1\cdots w_n$, where $w_0=0$.
Let $\asc(w)$ be the number of ascents of $w$.
Therefore, we obtain
$$2^nxA_n(x)=\sum_{w\in\msp_n}x^{\asc(w)}.$$

\begin{ex}
The following decorated permutations are generated from $\widehat{3}~\widehat{1}4~2$:
\begin{align*}
&\widehat{3}~\widehat{1}425,~\widehat{3}~\widehat{1}4~2~\widehat{5}, ~\widehat{3}~\widehat{1}4~5~2, ~\widehat{3}~\widehat{1}4~\circled{5}~2,\widehat{3}~\widehat{1}542,\\ &~\widehat{3}~\widehat{1}~\circled{5}42,~\widehat{3}~\widehat{5}~\widehat{1}42,~\widehat{3}~\circled{$\widehat{5}$}~\widehat{1}~4~2,
\widehat{5}~\widehat{3}~\widehat{1}42,~\circled{$\widehat{5}$}~\widehat{3}~\widehat{1}42.
\end{align*}
\end{ex}
\begin{ex}
We have $\msp_2=\{12,1\widehat{2},\widehat{1}2,\widehat{1}~\widehat{2},21,\circled{2}1,\widehat{2}~\widehat{1},\circled{$\widehat{2}$}~\widehat{1}\}$.
\end{ex}

Let $\I_{n,k}$ be the set of subsets of $[n]$ with cardinality $k$. Let $\HHat(w)$ be the set of entries of $w$ with hats and let $\hhat(w)=\#\HHat(w)$. Let $\varphi(\HHat(w))$ be a subset of $[n]$ obtained from $\HHat(w)$ by deleting all hats and circles of all entries of $\HHat(w)$.
For example, if $w=\circled{$\widehat{5}$}~\widehat{3}~\widehat{1}42$, then $\HHat(w)=\{\widehat{1},\widehat{3},\circled{$\widehat{5}$}\}$ and $\varphi(\HHat(w))=\{1,3,5\}$.
We define
\begin{align*}
\PM_{n,k}&=\{(S_1,S_2,I_{n,k}): S_1\in \mm_{2k},~S_2\in\mm_{2n-2k},~I_{n,k}\in \I_{n,k}\},\\
\msp_{n,k}&=\{w\in\msp_n: \hhat(w)=k\}.
\end{align*}
In this subsection, we always assume that the weight of $w\in\msp_{n,k}$ is $x^{\asc(w)}$ and that of the pair of matchings $(S_1,S_2)$ is $x^{\el(S_1)+\el(S_2)}$.

Now we start to construct a bijection, denoted by $\Phi$, between $\msp_{n,k}$ and
$\PM_{n,k}$. When $n=1$,
set $\Phi(1)=(\emptyset,(12),\emptyset)$ and $\Phi(\widehat{1})=((12),\emptyset,\{1\})$.
This gives a bijection between $\msp_{1,k}$ and $\PM_{1,k}$.
When $n=2$, the bijection between $\msp_{2,k}$ and $\PM_{2,k}$ is given as follows:
\begin{align*}
\Phi(12)&=(\emptyset,(12)(34),\emptyset), ~\Phi(21)=(\emptyset,(13)(24),\emptyset)\\
\Phi(\circled{2}1)&=(\emptyset,(14)(23),\emptyset),~\Phi(\widehat{1}{2})=((12),(12),\{1\}),\\
\Phi(1\widehat{2})&=((12),(12),\{2\}),~\Phi(\widehat{1}~\widehat{2})=((12)(34),\emptyset,\{1,2\}),\\
\Phi(\widehat{2}~\overline{1})&=((13)(24),\emptyset,\{1,2\}),~\Phi(\circled{$\widehat{2}$}~\widehat{1})=((14)(23),\emptyset,\{1,2\}).
\end{align*}

Suppose $\Phi$ is a bijection between $\msp_{m-1,k}$ and $\PM_{m-1,k}$ for all $k$, where $m\geq 3$.
Assume that $w=w_1w_2\cdots w_{m-1}\in\msp_{m-1,k}$, $\asc(w)=i+j$ and $\HHat(w)=\{w_{i_1},w_{i_2},\ldots,w_{i_k}\}$.
Let $\Phi(w)=(S_1,S_2,I_{m-1,k})$, where $S_1\in\mm_{2k},~S_2\in\mm_{2m-2k-2},~I_{m-1,k}=\varphi(\HHat(w)),~\el(S_1)=i,~\el(S_2)=j$.

Consider the case $n=m$.
Let $w'$ be a decorated permutation generated from $w$.
We first distinguish two cases:
If $w'=wm$, then let $\Phi(w')=(S_1,S_2(2m-2k-1,2m-2k),I_{m,k})$, where $I_{m,k}=\varphi(\HHat(w))$;
If $w'=w\widehat{m}$, then let $\Phi(w')=(S_1(2k+1,2k+2),S_2,I_{m,k+1})$, where $I_{m,k+1}=\varphi(\HHat(w))\cup\{m\}$.

Now let $\ell_1\ell_2$ be a consecutive subword of $w$.
Firstly, suppose that $\ell_2$ with no hat. We say $\ell_2$ is a {\it unhat-ascent-top} (resp. {\it unhat-descent-bottom}) if $\varphi(\ell_1)<\varphi(\ell_2)$ (resp. $\varphi(\ell_1)>\varphi(\ell_2)$).
Consider the following two cases:
\begin{enumerate}
  \item [\rm ($c_1$)] If $w'=\cdots\ell_1 m\ell_2$ (resp. $w'=\cdots\ell_1\circled{m} \ell_2\cdots$) and $\ell_2$ is the $p$th unhat-ascent-top of $w$, then let $\Phi(w')=(S_1,S'_2,I_{m,k})$, where $I_{m,k}=\varphi(\HHat(w))$ and $S'_2$ is obtained from $S_2$ by
  replacing $p$th marked block $(a,b)$ by two blocks $(a,2m-2k-1),(b,2m-2k)$ (resp. $(a,2m-2k),(b,2m-2k-1)$).
  \item [\rm ($c_2$)] If $w'=\cdots\ell_1m \ell_2\cdots$ (resp. $w'=\cdots\ell_1\circled{m} \ell_2\cdots$) and $\ell_2$ is the $p$th unhat-descent-bottom of $w$, then let $\Phi(w')=(S_1,S'_2,I_{m,k})$, where $I_{m,k}=\varphi(\HHat(w))$ and $S'_2$ is obtained from $S_2$ by replacing $p$th unmarked block $(a,b)$ by two blocks $(a,2m-2k-1),(b,2m-2k)$ (resp. $(a,2m-2k),(b,2m-2k-1)$).
\end{enumerate}

Secondly, suppose that $\ell_2$ with a hat. We say $\ell_2$ is a {\it hat-ascent-top} (resp. {\it hat-descent-bottom}) if $\varphi(\ell_1)<\varphi(\ell_2)$ (resp. $\varphi(\ell_1)>\varphi(\ell_2)$).
Consider the following two cases:
\begin{enumerate}
\item [\rm ($c_1$)] If $w'=\cdots\ell\widehat{m}~{\ell_2}\cdots$ (resp. $w'=\cdots\ell_1\circled{$\widehat{m}$}~{\ell_2}\cdots$) and $\ell_2$ is
the $p$th hat-ascent-top of $w$, then let $\Phi_1(w')=(S'_1,S_2,I_{m,{k+1}})$, where $I_{m,{k+1}}=\varphi(\HHat(w))\cup\{m\}$ and $S'_1$ is obtained from $S_1$ by replacing the $p$th marked block $(a,b)$ by two blocks $(a,2k+1),(b,2k+2)$ (resp. $(a,2k+2),(b,2k+1)$).
\item [\rm ($c_2$)] If $w'=\cdots\ell\widehat{m}~{\ell_2}\cdots$ (resp. $w'=\cdots\ell_1\circled{$\widehat{m}$}~{\ell_2}\cdots$) and $\ell_2$ is
the $p$th hat-descent-bottom of $w$,
then let $\Phi(w')=(S'_1,S_2,I_{m,{k+1}})$, where $I_{m,{k+1}}=\varphi(\HHat(w))\cup\{m\}$ and $S_1'$ is obtained from $S_1$ by replacing the $p$th
unmarked block $(a,b)$ by two blocks $(a,2k+1),(b,2k+2)$ (resp. $(a,2k+2),(b,2k+1)$).
\end{enumerate}

After the above step, we write the obtained perfect matchings in standard form. Suppose that $w\in\msp_{n,k}$ and $\Phi(w)=(S_1,S_2,I_{n,k})$.
Then $\asc(w)=i+j$ if and only if $\el(S_1)+\el(S_2)=i+j$.
By induction, we see that $\Phi$ is the desired bijection between $\PM_{n,k}$ to $\msp_{n,k}$ for all $k$,
which also gives a constructive proof of~\eqref{NnxAnx}.

\begin{ex}
Let $w=\widehat{3}~\widehat{1}42\circled{$\widehat{6}$}\widehat{5}\in \msp_{6,4}$.
The correspondence between $w$ and $\Phi(w)$ is built up as follows:
\begin{align*}
\widehat{1}&\Leftrightarrow ((12),\emptyset,\{1\});\\
\widehat{1}2&\Leftrightarrow ((12),(12),\{1\});\\
\widehat{3}~\widehat{1}2&\Leftrightarrow ((13)(24),(12),\{1,3\});\\
\widehat{3}~\widehat{1}42&\Leftrightarrow ((13)(24),(13)(24),\{1,3\});\\
\widehat{3}~\widehat{1}42\widehat{5}&\Leftrightarrow ((13)(24)(5,6),(13)(24),\{1,3,5\});\\
\widehat{3}~\widehat{1}42\circled{$\widehat{6}$}\widehat{5}&\Leftrightarrow ((13)(24)(5,8)(6,7),(13)(24),\{1,3,5,6\}).
\end{align*}
\end{ex}

\subsection{Signed permutations and pairs of perfect matchings}
\hspace*{\parindent}

In this subsection, we shall write signed permutations of $\mbn$ as $\pi=\pi(0)\pi(1)\pi(2)\cdots\pi(n)$, where
some elements are associated with the minus sign and $\pi(0)=0$. As usual, we denote by $\overline{i}$ the negative element $-i$.
For $\pi\in\mbn$, let
$\RLMIN(\pi)=\{\pi(i): |\pi(i)|<|\pi(j)| ~\textrm{for all $j>i$}\}$.
For example, $\RLMIN(\overline{3}~\overline{1}42\overline{6}7\overline{5})=\{\overline{1},2,\overline{5}\}$.
Let $\rlmin(\pi)=\#\RLMIN(\pi)$. It is clear that if $\pi\in\msn$, then $\rlmin(\pi)$ is the number of {\it right-to-left minima} of $\pi$.
Thus
$$\sum_{\pi\in\mbn}x^{\rlmin(\pi)}=2^n\sum_{\pi\in\msn}x^{\rlmin(\pi)} =2^nx(x+1)(x+2)\cdots(x+n-1)\quad\textrm{for $n\ge 1$}.$$

\begin{definition}
A block of $\pi$ is a maximal subsequence of consecutive elements of $\pi$ ending with $\pi(i)\in\RLMIN(\pi)$ and not contain any other element of $\RLMIN(\pi)$.
\end{definition}
It is clear that any $\pi$ has a unique decomposition as a sequence of its blocks.
If $\rlmin(\pi)=k$, then we write $\pi\mapsto B_1B_2\cdots B_k$,
where $B_i$ is $i$th block of $\pi$.
A {\it bar-block} (resp.~{\it unbar-block}) is a block ending with a negative (resp. positive) element.
Let $\BBar(\pi)$ be the union of elements of bar-blocks of $\pi$ and let $\bbar(\pi)=\#\BBar(\pi)$.
We define a map $\theta$ by
$$\theta(\BBar(\pi))=\{|\pi(i)|: \pi(i)\in \BBar(\pi)\}.$$
Set $\NBar(\pi)=[n]/\BBar(\pi)$.
For example, if $\pi=\overline{3}~\overline{1}42\overline{6}7\overline{5}$, then
$\pi\mapsto[\overline{3}~\overline{1}][42][\overline{6}7\overline{5}]$, $[\overline{3}~\overline{1}]$ and $[\overline{6}7\overline{5}]$ are bar-blocks of $\pi$, $\BBar(\pi)=\{\overline{6},\overline{5},\overline{3},\overline{1},7\},\bbar(\pi)=5$, $\theta(\BBar(\pi))=\{1,3,5,6,7\}$.
and $\NBar(\pi)=\{2,4\}$.

Let
${\mbp}_{n,k}=\{\pi\in\mbn: \bbar(\pi)=k\}$ and let
$${\BM}_{n,k}=\{(T_1,T_2,I_{n,k}): T_1\in \mm_{2k},~T_2\in\mm_{2n-2k},~I_{n,k}\in \I_{n,k}\},$$
where $\I_{n,k}$ is the set of subsets of $[n]$ with cardinality $k$.
In this subsection, we always assume that the weight of $\pi\in\mbp_{n,k}$ is $x^{\des_B(\pi)}$ and that of the pair of matchings $(T_1,T_2)$ is $x^{\el(T_1)+\ol(T_2)}$.

Along the same lines as the proof of~\eqref{NnxAnx}, we start to construct a bijection, denoted by $\Psi$, between ${\mbp}_{n,k}$ and
${\BM}_{n,k}$. When $n=1$,
set $\Psi(1)=(\emptyset,(12),\emptyset)$ and $\Psi(\overline{1})=((12),\emptyset,\{1\})$.
This gives a bijection between $\mbp_{1,k}$ and $\BM_{1,k}$.
When $n=2$, the bijection $\Psi$ between ${\mbp}_{2,k}$ and ${\BM}_{2,k}$ is given as follows:
\begin{align*}
\Psi(12)&=(\emptyset,(12)(34),\emptyset),~\Psi(21)=(\emptyset,(13)(24),\emptyset)\\
\Psi(\overline{2}1)&=(\emptyset,(14)(23),\emptyset),~\Psi(\overline{1}{2})=((12),(12),\{1\}),\\
\Psi(1\overline{2})&=((12),(12),\{2\}),~\Psi(\overline{1}~\overline{2})=((12)(34),\emptyset,\{1,2\}),\\
\Psi(2\overline{1})&=((13)(24),\emptyset,\{1,2\}),~\Psi(\overline{2}~\overline{1})=((14)(23),\emptyset,\{1,2\}).
\end{align*}

Suppose $\Psi$ is a bijection between $\mbp_{m-1,k}$ and $\BM_{m-1,k}$ for all $k$, where $m\geq 3$.
Assume that $\pi=\pi(1)\pi(2)\cdots \pi(m-1)\in{\mbp}_{m-1,k}$, $\des_B(\pi)=i+j$ and $\BBar(\pi)=\{\pi(i_1),\pi(i_2)\ldots,\pi(i_k)\}$.
Let $\Psi(\pi)=(T_1,T_2,I_{m-1,k})$, where $T_1\in\mm_{2k},~T_2\in\mm_{2m-2k-2},~I_{m-1,k}=\theta(\BBar(\pi)),~\el(T_1)=i,~\ol(T_2)=j$.
Consider the case $n=m$. Let $\pi'$ be obtained from $\pi$ by inserting
the entry $m$ (resp. $\overline{m}$) into $\pi$.
We first distinguish two cases: If $\pi'=\pi m$, then
let $$\Psi(\pi')=(T_1,T_2(2m-2k-1,2m-2k),I_{m,k}),$$ where $I_{m,k}=\theta(\BBar(\pi))$;
If ${\pi}'=\pi \overline{m}$, then let $\Psi({\pi}')=(T_1(2k+1,2k+2),T_2,I_{m,k+1})$, where $I_{m,k+1}=\theta(\BBar(\pi))\cup\{m\}$.

For $0\leq i\leq m-2$, consider the consecutive subword $\pi(i)\pi(i+1)$ of $\pi$.
Firstly, suppose that $\pi(i+1)\in \NBar(\pi)$. We say $\pi(i+1)$ is a {\it unbar-ascent-top} (resp. {\it unbar-descent-bottom}) if $\pi(i)<\pi(i+1)$ (resp. $\pi(i)>\pi(i+1)$).
Consider the following two cases:
\begin{enumerate}
  \item [\rm ($c_1$)] If $\pi'=\cdots\pi(i)m\pi(i+1)\cdots$ (resp. $\pi'=\cdots\pi(i)\overline{m}\pi(i+1)\cdots$) and $\pi(i+1)$ is the $p$th unbar-ascent-top of $\pi$, then let $\Psi(\pi')=(T_1,T_2',I_{m,k})$, where $T_2'$ is obtained from $T_2$ by replacing the $p$th marked block $(a,b)$ by two blocks $(a,2m-2k-1),(b,2m-2k)$ (resp. $(a,2m-2k),(b,2m-2k-1)$) and $I_{m,k}=\theta(\BBar(\pi))$.
  \item [\rm ($c_2$)] If $\pi'=\cdots\pi(i)m\pi(i+1)\cdots$ (resp. $\pi'=\cdots\pi(i)\overline{m}\pi(i+1)\cdots$) and $\pi(i+1)$ is the $p$th  unbar-descent-bottom of $\pi$, then let $\Psi(\pi')=(T_1,T_2',I_{m,k})$, where $T_2'$ is obtained from $T_2$ by replacing the $p$th unmarked block $(a,b)$ by two blocks $(a,2m-2k-1),(b,2m-2k)$ (resp. $(a,2m-2k),(b,2m-2k-1)$) and $I_{m,k}=\theta(\BBar(\pi))$.
\end{enumerate}

Secondly, suppose that $\pi(i+1)\in \BBar(\pi)$. We say $\pi(i+1)$ is a {\it bar-ascent-top} (resp. {\it bar-descent-bottom}) if $\pi(i)<\pi(i+1)$ (resp. $\pi(i)>\pi(i+1)$).
Consider the following two cases:
\begin{enumerate}
  \item [\rm ($c_1$)] If $\pi'=\cdots\pi(i)m\pi(i+1)\cdots$ (resp. $\pi'=\cdots\pi(i)\overline{m}\pi(i+1)\cdots$) and $\pi(i+1)$ is the $p$th bar-ascent-top of $\pi$, then let $\Psi(\pi')=(T_1',T_2,I_{m,{k+1}})$, where $T_1'$ is obtained from $T_1$ by replacing the $p$th unmarked block $(a,b)$ by two blocks $(a,2k+1),(b,2k+2)$ (resp. $(a,2k+2),(b,2k+1)$) and $I_{m,{k+1}}=\theta(\BBar(\pi))\cup\{m\}$.
  \item [\rm ($c_2$)] If $\pi'=\cdots\pi(i)m\pi(i+1)\cdots$ (resp. $\pi'=\cdots\pi(i)\overline{m}\pi(i+1)\cdots$) and $\pi(i+1)$ is the $p$th  bar-descent-bottom of $\pi$, then let $\Psi(\pi')=(T_1',T_2,I_{m,{k+1}})$, where $T_1'$ is obtained from $T_1$ by replacing the $p$th marked block $(a,b)$ by two blocks $(a,2k+1),(b,2k+2)$ (resp. $(a,2k+2),(b,2k+1)$) and $I_{m,{k+1}}=\theta(\BBar(\pi))\cup\{m\}$.
\end{enumerate}

After the above step, we write the obtained perfect matchings in standard form. Suppose that $\pi\in\mbp_{n,k}$ and $\Psi(\pi)=(T_1,T_2,I_{n,k})$.
Then $\des_B(\pi)=i+j$ if and only if $\el(T_1)+\ol(T_2)=i+j$.
By induction, we see that $\Psi$ is the desired bijection between $\mbp_{n,k}$ to $\BM_{n,k}$ for all $k$,
which also gives a constructive proof of~\eqref{NnxBnx}.

\begin{ex}
Let $\pi=\overline{3}~\overline{1}42\overline{6}7\overline{5}$.
The correspondence between $\pi$ and $\Psi(\pi)$ is built up as follows:
\begin{align*}
\overline{1}&\Leftrightarrow ((12),\emptyset,\{1\});\\
\overline{1}2&\Leftrightarrow ((12),(12),\{1\});\\
\overline{3}~\overline{1}2&\Leftrightarrow ((14)(23),(12),\{1,3\});\\
\overline{3}~\overline{1}42&\Leftrightarrow ((14)(23),(13)(24),\{1,3\});\\
\overline{3}~\overline{1}42\overline{5}&\Leftrightarrow ((14)(23)(5,6),(13)(24),\{1,3,5\});\\
\overline{3}~\overline{1}42\overline{6}~\overline{5}&\Leftrightarrow ((13)(24)(5,8)(6,7),(13)(24),\{1,3,5,6\});\\
\overline{3}~\overline{1}42\overline{6}7\overline{5}&\Leftrightarrow ((13)(24)(5,8)(6,9)(7,10),(13)(24),\{1,3,5,6,7\}).
\end{align*}
\end{ex}
\section{The String permutations of the second kind}\label{Section03}
Stirling permutations were introduced by Gessel and Stanley~\cite{Gessel78}. Let $[n]_2=\{1,1,2,2\ldots,n,n\}$.
A \emph{Stirling permutation} of order $n$ is a permutation of the multiset $[n]_2$ such that
every element between the two occurrences of $i$ is greater than $i$ for each $i\in [n]$.
For example, $\mq_2=\{1122,1221,2211\}$. Let $\sigma_1\sigma_2\cdots\sigma_{2n}\in\mqn$.
An index $i$ is a {\it descent} of $\sigma$ if $\sigma_i>\sigma_{i+1}$ or $i=2n$.
Let $C(n,k)$ be the number of Stirling permutations of $[n]_2$ with $k$ descents.
Following~\cite[Eq.~(6)]{Gessel78}, the numbers $C(n,k)$ satisfy the recurrence relation
\begin{equation}\label{Cnk-recu}
C(n,k)=kC(n-1,k)+(2n-k)C(n-1,k-1)
\end{equation}
for $n\geq2$, with the initial conditions $C(1,1)=1$ and $C(1,0)=0$.
The {\it second-order Eulerian polynomial} is defined by
$$C_n(x)=\sum_{i=1}^nC(n,k)x^k.$$

In recent years, there has been much work on Stirling permutations (see~\cite{Bona08,Dzhuma14,Janson11,Leon,Ma15,Remmel14}).
In particular, B\'ona~\cite{Bona08} introduced the plateau statistic on Stirling permutations, and proved that descents
and plateaus have the same distribution over $\mqn$. Given $\sigma\in\mqn$, the index $i$ is called a {\it plateau} if $\sigma_i=\sigma_{i+1}$.
We say that an index $i\in [2n-1]$ is an \emph{ascent plateau} if $\sigma_{i-1}<\sigma_i=\sigma_{i+1}$, where $\sigma_0=0$.
Let $\ap(\sigma)$ be the number of the ascent plateaus of $\sigma$. For example, $\ap(\textbf{2}211\textbf{3}3)=2$.
Very recently, we present a combinatorial proof of the following identity (see~\cite[Theorem 3]{MaYeh}):
\begin{equation}\label{Nnx-AP}
\sum_{\sigma\in\mqn}x^{\ap(\sigma)}=\sum_{\m\in\mmn}x^{\el(\m)}.
\end{equation}
Motivated by~\eqref{N2xzAx2z} and~\eqref{Nnx-AP}, we shall introduce Stirling permutations of the second kind.

Let $[k]^n$ denote the set of words of length $n$ in the alphabet $[k]$. For $\omega=\omega_1\omega_2\cdots\omega_n\in [k]^n$,
the reduction of $\omega$, denoted by $\redd(\omega)$, is the unique word of length $n$ obtained by replacing
the $i$th smallest entry by $i$. For example, $\redd(33224547)=22113435$.

\begin{definition}\label{def07}
A permutation $\sigma$ of the multiset $[n]_2$ is a {\it Stirling permutation of the second kind} of order $n$ whenever $\sigma$ can be written as a nonempty disjoint union of its distinct cycles and $\sigma$ has a standard cycle form satisfying the following conditions:
\begin{itemize}
  \item [\rm ($i$)] For each $i\in [n]$, the two copies of $i$ appear in exactly one cycle;
  \item [\rm ($ii$)] Each cycle is written with one of its smallest entry first and the cycles are written in increasing order of
their smallest entry;
  \item [\rm ($iii$)] The reduction of the word formed by all entries of each cycle is a Stirling permutation. In other words, if $(c_1,c_2,\ldots,c_{2k})$ is a cycle of $\sigma$, then $\redd(c_1c_2\cdots c_{2k})\in \mq_k$.
\end{itemize}
\end{definition}

Let $\mqn^2$ denote the set of Stirling permutations of the second kind of order $n$.
In the following discussion, we always write $\sigma\in\mqn^2$ in standard cycle form.
\begin{ex}
\begin{align*}
 \mq_1^2&=\{(11)\},  \mq_2^2=\{(11)(22),(1122),(1221)\}, \\
 \mq_3^2&=\{(11)(22)(33),(11)(2233),(11)(2332),(1133)(22),(1331)(22),(1122)(33),(112233),\\
 &(112332),(113322),(133122),(1221)(33),(122133),(122331),(123321),(133221)\}.
\end{align*}
\end{ex}

Let $(c_1,c_2,\ldots,c_{2k})$ be a cycle of $\sigma$. An entry $c_i$ is called a {\it cycle plateau} (resp. {\it cycle ascent}) if
$c_i=c_{i+1}$ (resp. $c_i<c_{i+1}$), where $1\leq i<2k$. Let $\cplat(\pi)$ and $\casc(\pi)$ be the number of cycle plateaus and cycle ascents of $\pi$, respectively. For example, $\cplat((1\textbf{2}21)(\textbf{3}3))=2$ and $\casc((\textbf{1}221)(33))=1$.
Now we present a dual result of~\cite[Proposition~1]{Bona08}.
\begin{prop}
For $n\geq 1$, we have $$C_n(x)=\sum_{\pi\in\mqn^2}x^{\cplat(\pi)}=\sum_{\pi\in\mqn^2}x^{\casc(\pi)+1}.$$
\end{prop}
\begin{proof}
There are two ways in which a permutation $\sigma'\in \mqn^2$ with $k$ cycle plateaus can be obtained from a
permutation $\sigma\in\mq_{n-1}^2$.
If $\cplat(\sigma)=k$, then we can put the two copies of $n$ right after a cycle plateau of $\sigma$.
This gives $k$ possibilities.
If $\cplat(\sigma)=k-1$, then we can append a new cycle $(nn)$ right after $\sigma$ or insert the two copies of $n$ into any of the remaining $2n-2-(k-1)=2n-k-1$ positions. This gives $2n-k$ possibilities.
Comparing with~\eqref{Cnk-recu}, this completes the proof of $C_n(x)=\sum_{\pi\in\mqn^2}x^{\cplat(\pi)}$. Along the same lines, one can easily prove the assertion for cycle ascents. This completes the proof.
\end{proof}

Let $(c_1,c_2,\ldots,c_{2k})$ be a cycle of $\sigma$, where $k\geq2$.
An entry $c_i$ is called a {\it cycle ascent plateau} if
$c_{i-1}<c_{i}=c_{i+1}$, where $2\leq i\leq2k-1$.
Denote by $\caplat(\sigma)$ (resp. $\cyc(\sigma)$) the number of cycle ascent plateaus (resp. cycles) of $\sigma$.
For example, $\caplat((1\textbf{2}21)({3}3))=1$.
We define $$Q_n(x,q)=\sum_{\sigma\in\mqn^2}x^{\caplat(\sigma)}q^{\cyc(\sigma)},$$
$$Q(x,q;z)=1+\sum_{n\geq 1}Q_n(x,q)\frac{z^n}{n!}.$$

Our main result of this section is the following.
\begin{theorem}
The polynomials $Q_n(x,q)$ satisfy the recurrence relation
\begin{equation}\label{Qnxq-recu}
Q_{n+1}(x,q)=(q+2nx)Q_n(x,q)+2x(1-x)\frac{\partial}{\partial x}Q_n(x,q)
\end{equation}
for $n\geq 0$, with the initial condition $Q_0(x)=1$.
Moreover,
\begin{equation}\label{Qxqz}
Q(x,q;z)=\left(\sqrt{\frac{x-1}{x-e^{2z(x-1)}}}\right)^q.
\end{equation}
\end{theorem}
\begin{proof}
Given $\sigma\in\mqn^2$.
Let $\sigma_i$ be an element of $\mq_{n+1}^2$ obtained from $\sigma$ by inserting the two copies of
$n+1$, in the standard cycle decomposition of $\sigma$, right after $i\in [n]$ or as a new cycle $(n+1,n+1)$ if $i=n+1$.
It is clear that
$$ \cyc(\sigma_i)=\left\{
              \begin{array}{ll}
                \cyc(\sigma), & \hbox{if $i\in [n]$;} \\
                \cyc(\sigma)+1, & \hbox{if $i=n+1$.}
              \end{array}
            \right.
$$
Therefore, we have
\begin{align*}
Q_{n+1}(x,q)&=\sum_{\pi\in\mqnn^2}x^{\caplat(\pi)}q^{\cyc(\pi)}\\
&=\sum_{i=1}^{n+1}\sum_{\sigma_i\in\mqn^2}x^{\caplat(\sigma_i)}q^{\cyc(\sigma_i)}\\
&=\sum_{\sigma\in\mqn^2}x^{\caplat(\sigma)}q^{\cyc(\sigma)+1}+\sum_{i=1}^{n}\sum_{\sigma_i\in\mqn^2}x^{\caplat(\sigma_i)}q^{\cyc(\sigma_i)}\\
&=qQ_n(x,q)+\sum_{\sigma\in\mqn^2}(2\caplat(\sigma)x^{\caplat(\sigma)}+(2n-2\caplat(\sigma))x^{\caplat(\sigma)+1})q^{\cyc(\sigma)}
\end{align*}
and~\eqref{Qnxq-recu} follows.
By rewriting~\eqref{Qnxq-recu} in terms of the exponential generating function $Q(x,q;z)$, we have
\begin{equation}\label{Qxz-pde}
(1-2xz)\frac{\partial}{\partial z}Q(x,q;z)=qQ(x,q;z)+2x(1-x)\frac{\partial}{\partial x}Q(x,q;z).
\end{equation}
It is routine to check that the generating function
$$\widetilde{Q}(x,q;z)=\left(\sqrt{\frac{x-1}{x-e^{2z(x-1)}}}\right)^q$$
satisfies~\eqref{Qxz-pde}. Also, this generating function gives $\widetilde{Q}(x,q;0)=1,\widetilde{Q}(x,0;z)=1$
and $\widetilde{Q}(0,q;z)=e^{qz}$. Hence $Q(x,q;z)=\widetilde{Q}(x,q;z)$.
\end{proof}

Combining~\eqref{N2xt02} and~\eqref{Qxqz}, we get $Q(x,q;z)={M^q(x,z)}$.
Thus $Q_n(x,1)=M_n(x)$.
Moreover, it follows from~\eqref{Qnxq-recu} that
$Q_{n+1}(1,q)=(q+2n)Q_n(1,q)$.
So the following corollary is immediate.
\begin{cor}
For $n\geq 1$, we have
$$\sum_{\sigma\in\mqn^2}q^{\cyc(\sigma)}=q(q+2)\cdots (q+2n-2).$$
\end{cor}

We now introduce a statistic on $\mqn$ that is equidistributed with the cycle statistic on $\mqn^2$.
Denote by $[i,j]$ the interval of all integers between $i$ and $j$, where $i\leq j$. In particular, when $i=j$, we denote by $[i]$ the singleton interval.
For $\sigma=\sigma_1\sigma_2\cdots\sigma_{2n}\in\mqn$,
if $\sigma_i>\sigma_{i+1}$, all of the different elements before $\sigma_{i+1}$ appear a second time and all of these different entries construct an interval, then this interval is called a {\it descent interval} of $\sigma$, where $i\in[2n]$ and $\sigma_{2n+1}=0$. For example, if $\sigma=44223311$ and $\tau=113322$, then the descent intervals of $\sigma$ are $[4],[2,4]$ and $[1,4]$, and that of $\tau$ is $[1,3]$.

Let $\desi(\sigma)$ be the number of descent intervals of $\sigma$.
Define
$$L_n(q)=\sum_{\sigma\in\mqn}q^{\desi(\sigma)}.$$
Let $\sigma^{(i)}\in\mq_{n+1}$ be obtained from $\sigma\in\mqn$ by inserting two copies of $n+1$ before $\sigma_i$.
It is evident that
$$ \desi(\sigma^{(i)})=\left\{
              \begin{array}{ll}
                \desi(\sigma)+1, & \hbox{if $i=1$;} \\
                \desi(\sigma), & \hbox{otherwise.}
              \end{array}
            \right.
$$
Thus
\begin{equation*}\label{Lnq-recu}
L_{n+1}(q)=(q+2n)L_n(q).
\end{equation*}
The following result is immediate.
\begin{prop}
For $n\geq 1$, we have $$\sum_{\sigma\in\mqn}q^{\desi(\sigma)}=\sum_{\sigma\in\mqn^2}q^{\cyc(\sigma)}.$$
\end{prop}

Let $A_n(x)=\sum_{k=0}^{n-1}\Eulerian{n}{k}x^k$. The numbers $\Eulerian{n}{k}$ are called Eulerian numbers and satisfy the recurrence relation
$$\Eulerian{n}{k}=(k+1)\Eulerian{n-1}{k}+(n-k)\Eulerian{n-1}{k-1},$$
with initial conditions $\Eulerian{0}{0}=1$ and $\Eulerian{0}{k}$ for $k\geq 1$ (see~\cite{Bona12,Bre00} for instance).
Hence
\begin{equation}\label{Anx-recu}
A_{n+1}(x)=(1+nx)A_n(x)+x(1-x)A_n'(x),
\end{equation}
Let $\mcq_n$ denote the set of Stirling permutations of $\mqn^2$ with
only one cycle, which can be named as the set of {\it cyclic Stirling permutations}. Define $$Y_n(x)=\sum_{\sigma\in\mcq_n}x^{\caplat(\sigma)}.$$
Comparing~\eqref{Qnxq-recu} with~\eqref{Anx-recu}, we get the following corollary.
\begin{cor}
For $n\geq 1$, we have
\begin{equation*}
Y_{n+1}(x)=2^nxA_n(x).
\end{equation*}
\end{cor}

\section{The distribution of cycle ascent plateaus and fixed points on $\mqn^2$}\label{Section04}
Given $\sigma\in\mqn^2$. Let the entry $k\in[n]$ be called a {\it fixed point} of $\sigma$ if $(kk)$ is a cycle of $\sigma$.
The number of fixed points of $\sigma$ is defined by
$$\fix(\sigma)=\#\{k\in[n]: (kk)~\textrm{is a cycle of $\sigma$}\}.$$
For example, $\fix((1133)(22))=1$.
Define
\begin{align*}
P_n(x,y,q)&=\sum_{\sigma\in\mqn^2}x^{\caplat(\sigma)}y^{\fix(\sigma)}q^{\cyc(\pi)},\\
P(x,y,q;z)&=\sum_{n\geq 0}P_n(x,y,q)\frac{z^n}{n!}.
\end{align*}
Now we present the main result of this section.
\begin{theorem}\label{thm-Pnxyq}
For $n\geq 1$, the polynomials $P_n(x,y,q)$ satisfy the recurrence relation
\begin{equation}\label{Pxyq-Ax}
P_{n+1}(x,y,q)=qyP_n(x,y,q)+qx\sum_{k=0}^{n-1}\binom{n}{k}P_k(x,y,q)2^{n-k}A_{n-k}(x),
\end{equation}
with the initial conditions $P_0(x,y,q)=1,P_1(x,y,q)=yq$. Moreover,
\begin{equation}\label{Pn-recurrence}
P_{n+1}(x,y,q)=(2nx+qy)P_n(x,y,q)+2x(1-x)\frac{\partial}{\partial x}P_n(x,y,q)+2x(1-y)\frac{\partial}{\partial y}P_n(x,y,q).
\end{equation}
Furthermore,
\begin{equation}\label{Pxyqz-explicit}
P(x,y,q;z)=e^{qz(y-1)}Q(x,q;z).
\end{equation}
\end{theorem}

In the following, we shall prove Theorem~\ref{thm-Pnxyq} by using context-free grammars.
For an alphabet $A$, let $\mathbb{Q}[[A]]$ be the rational commutative ring of formal power
series in monomials formed from letters in $A$. Following~\cite{Chen93}, a context-free grammar over
$A$ is a function $G: A\rightarrow \mathbb{Q}[[A]]$ that replace a letter in $A$ by a formal function over $A$.
The formal derivative $D$ is a linear operator defined with respect to a context-free grammar $G$. More precisely,
the derivative $D=D_G$: $\mathbb{Q}[[A]]\rightarrow \mathbb{Q}[[A]]$ is defined as follows:
for $x\in A$, we have $D(x)=G(x)$; for a monomial $u$ in $\mathbb{Q}[[A]]$, $D(u)$ is defined so that $D$ is a derivation,
and for a general element $q\in\mathbb{Q}[[A]]$, $D(q)$ is defined by linearity.
\begin{lemma}\label{lemma01}
If $A=\{a,b,c,d\}$ and
$G=\{a\rightarrow qab^2, b\rightarrow b^{-1}c^2d^2, c\rightarrow cd^2, d\rightarrow c^2d\}$,
then
\begin{equation}\label{Dna}
D^n(a)=a\sum_{\sigma\in\mqn^2}q^{\cyc(\sigma)}b^{2\fix(\sigma)}c^{2\caplat(\sigma)}d^{2n-2\fix(\sigma)-2\caplat(\sigma)}.
\end{equation}
\end{lemma}
\begin{proof}
Let $\mqn^2(i,j,k)=\{\sigma\in\mqn^2: \cyc(\sigma)=i,\fix(\sigma)=j,\caplat(\sigma)=k\}$.
Given $\sigma\in\mqn^2(i,j,k)$. We now introduce a labeling scheme for $\sigma$:
\begin{itemize}
  \item [\rm ($i$)] Put a superscript label $a$ at the end of $\sigma$ and a superscript $q$ before each cycle of $\sigma$;
  \item [\rm ($ii$)] If $k$ is a fixed point of $\sigma$, then we put a superscript label $b$ right after each $k$;
  \item [\rm ($iii$)] Put superscript labels $c$ immediately before and right after each cycle ascent plateau;
  \item [\rm ($iv$)] In each of the remaining positions
except the first position of each cycle, we put a superscript label $d$.
\end{itemize}

When $n=1$, we have $\mq_1^2(1,1,0)=\{^q(1^b1^b)^a\}$.
When $n=2$, we have $\mq_2^2(2,2,0)=\{^q(1^b1^b)^q(2^b2^b)^a\}$ and $\mq_2^2(1,0,1)=\{^q(1^d1^c2^c2^d)^a,~^q(1^c2^c2^d1^d)^a\}$.
Let $n=m$. Suppose we get all labeled permutations in $\mq_m^2(i,j,k)$ for all $i,j,k$, where $m\geq 2$. We consider the case $n=m+1$.
Let $\sigma'\in\mq_{m+1}^2$ be obtained from $\sigma\in\mq_{m}^2(i,j,k)$ by inserting two copies of the entry $m+1$ into $\sigma$.
Now we construct a correspondence, denoted by $\vartheta$, between $\sigma$ and $\sigma'$.
Consider the following cases:
\begin{itemize}
  \item [\rm ($c_1$)] If the two copies of $m+1$ are put at the end of $\sigma$ as a new cycle $((m+1)(m+1))$, then we leave all labels of $\sigma$ unchanged except the last cycle. In this case, the correspondence $\vartheta$ is defined by
$$\sigma=\cdots (\cdots)^a\xlongleftrightarrow{\vartheta}\sigma'=\cdots (\cdots)^q((m+1)^b(m+1)^b)^a,$$
which corresponds to the operation $a\rightarrow qab^2$. Moreover, $\sigma'\in\mq_{m+1}^2(i+1,j+1,k)$.
  \item [\rm ($c_2$)]If the two copies of $m+1$ are inserted to a position of $\sigma$ with label $b$, then $\vartheta$ corresponds to the operation $b\rightarrow b^{-1}c^2d^2$. In this case, $\sigma'\in\mq_{m+1}^2(i,j-1,k+1)$.
  \item [\rm ($c_3$)] If the two copies of $m+1$ are inserted to a position of $\sigma$ with label $c$, then $\vartheta$ corresponds to the operation $c\rightarrow cd^2$. In this case, $\sigma'\in\mq_{m+1}^2(i,j,k)$.
  \item [\rm ($c_4$)] If the two copies of $m+1$ are inserted to a position of $\sigma$ with label $d$, then $\vartheta$ corresponds to the operation $c\rightarrow c^2d$. In this case, $\sigma'\in\mq_{m+1}^2(i,j,k+1)$.
\end{itemize}
By induction, we see that $\vartheta$ is the desired correspondence between permutations in $\mq_m^2$ and $\mq_{m+1}^2$,
which also gives a constructive proof of~\eqref{Dna}.
\end{proof}

\begin{lemma}\label{lemma02}
If $A=\{b,c,d\}$ and
$G=\{b\rightarrow b^{-1}c^2d^2, c\rightarrow cd^2, d\rightarrow c^2d\}$,
then
\begin{equation*}
D^n(b^2)=2^n\sum_{k=0}^{n-1}\Eulerian{n}{k}c^{2k+2}d^{2n-2k}=2^nd^{2n}c^2A_n\left(\frac{c^2}{d^2}\right)~\textrm{for $n\ge 1$}.
\end{equation*}
\end{lemma}
\begin{proof}
Note that $D(b^2)=2c^2d^2$. Hence $D^n(b^2)=2D^{n-1}(c^2d^2)$ for $n\geq 1$.
Note that $D(c^2d^2)=2(c^2d^4+c^4d^2)$.
Assume that
$$D^n(b^2)=2^n\sum_{k=0}^{n-1}F(n,k)c^{2n-2k}d^{2k+2}.$$
Since
\begin{eqnarray*}
  D(D^n(b^2)) &=& 2^{n+1}\sum_{k=0}^{n-1}(n-k)F(n,k)c^{2n-2k}d^{2k+4}+2^{n+1}\sum_{k=0}^{n-1}(k+1)F(n,k)c^{2n-k+2}d^{2k+2},
\end{eqnarray*}
there follows
$$F(n+1,k)=(k+1)F(n,k)+(n-k+1)F(n,k-1).$$
We see that
the coefficients $F(n,k)$ satisfy the same recurrence relation and initial conditions as $\Eulerian{n}{k}$, so they agree.
\end{proof}

\noindent{\bf Proof of Theorem~\ref{thm-Pnxyq}:}\\
By Lemma~\ref{lemma01} and Lemma~\ref{lemma02}, we get
\begin{align*}
D^{n+1}(a)=&qD^n(ab^2)\\
          =&q\sum_{k=0}^n\binom{n}{k}D^k(a)D^{n-k}(b^2)\\
          =&qb^2D^n(a)+q\sum_{k=0}^{n-1}\binom{n}{k}D^k(a)2^{n-k}d^{2n-2k}c^2A_{n-k}\left(\frac{c^2}{d^2}\right).
\end{align*}
Taking $c^2=x,b^2=y$ and $d^2=1$ in both sides of the above identity, we immediately get~\eqref{Pxyq-Ax}.
Set $S_n(i,j,k)=\#\mqn^2(i,j,k)$. The following recurrence relation follows easily from the proof of Lemma~\ref{lemma01}:
$$S_{n+1}(i,j,k)=S_n(i-1,j-1,k)+2(j+1)S_n(i,j+1,k-1)+2kS_n(i,j,k)+2(n-j-k+1)S_n(i,j,k-1).$$
Multiplying both sides of the last recurrence relation by $q^iy^jx^k$ and summing for all $i,j,k$, we immediately get~\eqref{Pn-recurrence}.
Note that
\begin{align*}
P_{n}(x,y,q)&=\sum_{i=0}^n\binom{n}{i}(yq-q)^i\sum_{\sigma\in\mqn^2}x^{\caplat(\sigma)}q^{\cyc(\pi)}\\
            &=\sum_{i=0}^n\binom{n}{i}(yq-q)^iQ_{n-i}(x,q).
\end{align*}
Thus
$P(x,y,q;z)=e^{qz(y-1)}Q(x,q;z)$.
This completes the proof of Theorem~\ref{thm-Pnxyq}.

Given $\sigma\in\mqn^2$.
We say that $\sigma$ is a {\it Stirling derangement} if $\sigma$ has no fixed points.
Let $\mdqn$ be the set of Stirling derangements of $\mqn^2$.
Let $R_{n,k}(x,q)$ be the coefficient of $y^k$ in $P_n(x,y,q)$. Note that
$R_{n,0}$ is the corresponding enumerative polynomials on $\mdqn$. Set $R_n(x,q)=R_{n,0}(x,q)$.
Note that
\begin{align*}
R_{n,k}(x,q)&=\sum_{\substack{\sigma\in\mqn^2 \\ \fix(\sigma)=k}}x^{\caplat(\sigma)}q^{\cyc(\pi)}\\
         &=\binom{n}{k}q^k\sum_{\sigma\in \mdq_{n-k}}x^{\caplat(\sigma)}q^{\cyc(\pi)}\\
         &=\binom{n}{k}q^kR_{n-k}(x,q).
\end{align*}
Comparing the coefficients of both sides of~\eqref{Pn-recurrence}, we get the following result.
\begin{theorem}
For $n\geq 1$, the polynomials $R_n(x,q)$ satisfy the recurrence relation
\begin{equation}\label{Dnxq}
R_{n+1}(x,q)=2nxR_n(x,q)+2x(1-x)\frac{\partial}{\partial x}R_n(x,q)+2nxqR_{n-1}(x,q),
\end{equation}
with the initial conditions $R_1(x,q)=0,R_2(x,q)=2qx,R_3(x,q)=4qx(1+x)$.
\end{theorem}

Let $q_n=\#\mdqn$.
Then the following corollary is immediate.
\begin{cor}
For $n\geq 1$,
the numbers $q_n$ satisfy the recurrence relation
$q_{n+1}=2n(q_{n}+q_{n-1})$,
with the initial conditions $q_0=1,q_1=0$ and $q_2=2$.
\end{cor}
Note that $\#\mqn^2=\mqn=(2n-1)!!$. Then
$$\sum_{n\geq 0}\#\mqn^2\frac{z^n}{n!}=\frac{1}{\sqrt{1-2z}}.$$
Thus
$$\sum_{n\geq 0}q_n\frac{z^n}{n!}=\frac{e^{-z}}{\sqrt{1-2z}},$$
which can be easily proved by using the {\it exponential formula} (see~\cite[Theorem 3.50]{Bona12}).
It should be noted that $q_{n+1}$ is also the number of minimal number of 1-factors in a $2n$-connected graph having at least one 1-factor (see~\cite{Bollobas78}). It would be interesting to study the relationship between Stirling permutations of the second kind and $2n$-connected graphs.

In recent years, there has been much work on derangements polynomials of Coxeter groups (see~\cite{Chen09,Chow09,Kim01,Lin15,Zhang95} for instance).
For each $\pi\in\msn$,
an index $i$ is called {\it excedance} (resp. {\it anti-excedance}) if $\pi(i)>i$ (resp. $\pi(i)<i$).
Let $\exc(\pi)$ be the number of excedances of $\pi$.
A permutation $\pi\in\msn$ is a {\it derangement} if $\pi(i)\neq i$ for any $i\in [n]$.
Let $\mdn_n$ denote the set of derangements of $\msn$.
The {\it derangements polynomial} is defined by
$$d_n(x)=\sum_{\pi\in\mdn_n}x^{\exc(\pi)}.$$
Brenti~\cite[Proposition 5]{Bre90} derived that
\begin{equation}\label{dxz}
d(x,z)=\sum_{n\geq 0}d_n(x)\frac{z^n}{n!}=\frac{1-x}{e^{xz}-xe^z}.
\end{equation}

The {\it Stirling derangement polynomial} is defined by
$$R_n(x)= \sum_{\sigma\in\mdq_{n}}x^{\caplat(\sigma)}.$$
Let $$S(x,z)=\sum_{n\geq 0}R_n(x)\frac{z^n}{n!}.$$
Taking $y=0$ and $q=1$ in~\eqref{Pxyqz-explicit}, we have
\begin{equation}\label{Rxz}
S(x,z)=\sqrt{\frac{x-1}{xe^{2z}-e^{2xz}}}.
\end{equation}
Thus
$$S^2(x,z)=d(x,2z),$$
which is a dual result of~\eqref{N2xzAx2z}. Equivalently,
$$2^nd_n(x)=\sum_{k=0}^n\binom{n}{k}R_k(x)R_{n-k}(x).$$
Moreover, combining~\eqref{Dnxq},~\eqref{Rxz} and~\cite[Corollary 2.4]{Liu07}, we immediately get the following result.
\begin{prop}
For $n\geq 2$, the polynomial $R_n(x)$ is symmetric and has only simple real zeros.
\end{prop}

Let $\imath^2=\sqrt{-1}$.
From~\eqref{Rxz}, putting $x=-1$, we deduce the expression
\begin{equation}
S(-1,z)=\sqrt{\frac{2}{e^{2z}+e^{-2z}}}=\sqrt{\sec(2\imath z)}.
\end{equation}
Note that $\sec(z)$ is an even function.
Therefore, for $n\geq 1$, we have
$$ \sum_{\sigma\in\mdq_{n}}(-1)^{\caplat(\sigma)}=\left\{
              \begin{array}{ll}
                0, & \hbox{if $n=2k-1$;} \\
                (-1)^kh_k, & \hbox{if $n=2k$,}
              \end{array}
            \right.
$$
where the number $h_n$ is defined by the following series expansion:
$$\sqrt{\sec(2\imath z)}=\sum_{n\geq0}(-1)^nh_n\frac{z^{2n}}{(2n)!}.$$

The first few of the numbers $h_n$ are $h_0=1,h_1=2,h_2=28,h_3=1112,h_4=87568$.
It should be noted that the numbers $h_n$ also count permutations of $\msss_{4n}$ having the following properties:
\begin{enumerate}
  \item [(a)] The permutation can be written as a product of disjoint cycles with only two elements;
  \item [(b)] For $i\in [2n]$, indices $2i-1$ and $2i$ are either both excedances or both anti-excedances.
\end{enumerate}
For example, when $n=1$, there are only two permutations having the desired properties: $(1,3)(2,4)$ and $(14)(23)$.
This kind of permutations was introduced by Sukumar and Hodges~\cite{Sukumar07}.
\section{Concluding remarks}\label{Section-5}
A natural generalization of Stirling permutations is $k$-Stirling permutations.
Let $j^i$ denote the $i$ copies of $j$, where $i,j\geq 1$.
We call a permutation of the multiset
$\{1^k,2^k,\ldots,n^k\}$ a $k$-{\it Stirling permutation} of order $n$ if for each $i$, $1\leq i\leq n$,
all entries between the two occurrences of $i$ are at least $i$.
One can introduce $k$-Stirling permutations of the second kind along the same line as in Definition~\ref{def07}.


\begin{thebibliography}{14}
\bibitem{Bollobas78}
B. Bollob\'{a}s, \newblock The number of 1-factors in $2k$-connected graphs, \newblock {\em J. Combin. Theory Ser. B} 25 (1978), 363--366.


\bibitem{Bona08}
M. B\'ona, \newblock Real zeros and normal distribution for statistics on Stirling permutations defined by Gessel and Stanley, \newblock {\em SIAM J. Discrete Math.} 23 (2008/09), 401--406.

\bibitem{Bona12}
M. B\'ona, Combinatorics of Permutations, second edition, CRC Press, Boca Raton, FL, 2012.

\bibitem{Bre90}
F. Brenti, Unimodal polynomials arising from symmetric functions, \newblock {\em Proc. Amer.
Math. Soc.} 108 (1990), 1133--1141.


\bibitem{Bre94}
F. Brenti, \newblock $q$-Eulerian polynomials arising from Coxeter groups,
\newblock {\em European J. Combin.} 15 (1994), 417--441.
\bibitem{Bre00}
F. Brenti, \newblock A class of $q$-symmetric functions arising from plethysm,
\newblock {\em J. Combin. Theory Ser. A} 91 (2000), 137--170.

\bibitem{Chen93}
W.Y.C. Chen, \newblock Context-free grammars, differential operators and formal
power series. \newblock {\em Theoret. Comput. Sci.} 117 (1993), 113--129.


\bibitem{Chen09}
W.Y.C. Chen, R.L. Tang and A.F.Y. Zhao, \newblock Derangement polynomials and excedances of type $B$,
\newblock {\em Electron. J. Combin.} 16(2) (2009), \#R15.

\bibitem{Chen121}
W.Y.C. Chen, R.X.J. Hao and H.R.L. Yang, \newblock Context-free grammars and multivariate stable polynomials over Stirling permutations, \arxiv{1208.1420v2}.

\bibitem{Chen14}
W.Y.C. Chen, A.M. Fu, \newblock Context-free grammars for permutations and increasing trees, \arxiv{1408.1859}.
%
\bibitem{Chow09}
C.-O. Chow, On derangement polynomials of type $B$. II, \newblock {\em J. Combin. Theory Ser. A}
116 (2009), 816--830.


\bibitem{Dilks09}
K. Dilks, T.K. Petersen, J.R. Stembridge, \newblock Affine descents and the Steinberg torus, \newblock {\em Adv. in Appl. Math.} 42 (2009), 423--444.

\bibitem{Dumont96}
D. Dumont, Grammaires de William Chen et d\'erivations dans les arbres et
arborescences, \newblock {\em S\'em. Lothar. Combin.} 37 Art. B37a, (1996), pp.1--21.


\bibitem{Dzhuma14}
A. Dzhumadil'daev, D. Yeliussizov, Stirling permutations on multisets, \newblock {\em European J. Combin.} 36 (2014), 377--392.

\bibitem{Foata09}
D. Foata, G.-N. Han, The $q$-tangent and $q$-secant numbers via basic Eulerian polynomials, \newblock {\em Proc. Amer. Math. Soc.} 138 (2009), 385--393.



\bibitem{Gessel78}
I. Gessel and R.P. Stanley, \newblock Stirling polynomials, \newblock {\em J. Combin. Theory Ser.
A} 24 (1978), 25--33.

\bibitem{Leon}
R. S. Gonz\'{a}lez D'Le\'{o}n, \newblock A note on the $\gamma$-coefficients of the tree Eulerian polynomial,
\newblock {\em Electron. J. Combin.} 23(1) (2016), \#P1.20.


\bibitem{Janson11}
S. Janson, M. Kuba and A. Panholzer, \newblock  Generalized Stirling permutations, families of increasing trees and urn models,
\newblock {\em J. Combin. Theory Ser. A} 118 (2011), 94--114.

\bibitem{Kim01}
D. Kim, J. Zeng, A new decomposition of derangements, \newblock {\em J. Combin. Theory Ser. A} 96 (2001), 192--198.

\bibitem{Lin15}
Z. Lin, J. Zeng, The $\gamma$-positivity of basic Eulerian polynomials via group actions, \newblock {\em J. Combin. Theory Ser.
A} 135 (2015), 112--129.

\bibitem{Liu07}
L. Liu and Y. Wang, A unified approach to polynomial sequences with only real
zeros, \newblock {\em Adv. in Appl. Math.} 38 (2007), 542--560.
\bibitem{Ma13}
S.-M. Ma, A family of two-variable derivative polynomials for tangent and secant,
Electron. J. Combin. 20(1) (2013), \#P11.
%
\bibitem{Ma-EuJC}
S.-M. Ma, Some combinatorial arrays generated by context-free grammars, \newblock {\em European J. Combin.} 34 (2013), 1081--1091.

\bibitem{Ma15}
S.-M. Ma, T. Mansour, \newblock The $1/k$-Eulerian polynomials and $k$-Stirling permutations, \newblock {\em Discrete Math.} 338 (2015), 1468--1472.

\bibitem{MaYeh}
S.-M. Ma, Y.-N. Yeh, \newblock Stirling permutations, cycle structure of
permutations and perfect matchings, \newblock {\em Electron. J. Combin.} 22(4) (2015), \#P4.42.



\bibitem{Remmel14}
J.B. Remmel, A.T. Wilson, \newblock Block patterns in Stirling permutations, \arxiv{1402.3358}.

\bibitem{Savage1201}
C.D. Savage and M.J. Schuster, \newblock Ehrhart series of lecture hall polytopes and Eulerian polynomials for
inversion sequences, \newblock {\em J. Combin. Theory Ser. A} 119 (2012), 850--870.

\bibitem{Savage1202}
C.D. Savage and G. Viswanathan, \newblock The $1/k$-Eulerian polynomials, \newblock {\em Electron. J. Combin.},
19 (2012), \#P9.

\bibitem{Sloane}
N.J.A. Sloane, The On-Line Encyclopedia of Integer Sequences,
published electronically at http://oeis.org, 2010.

\bibitem{Sukumar07}
C.V. Sukumar and A. Hodges, Quantum algebras and parity-dependent spectra, \newblock {\em Proc. R. Soc. A} 463 (2007), 2415--2427.

\bibitem{Zhang95}
X.D. Zhang, On $q$-derangement polynomials, Combinatorics and Graph Theory 95,
Vol.~1 (Hefei), World Sci. Publishing, River Edge, NJ, 1995, pp. 462--465.

\end{thebibliography}
\end{document}